\documentclass[11pt]{amsart} 

\usepackage[utf8]{inputenc} 
\usepackage{amsmath}
\usepackage{amssymb}
\usepackage{mathrsfs}
\usepackage{mathtools}
\usepackage{mdwlist}
\usepackage{array}
\usepackage{url}

\def\ditemfirst#1#2{\begin{enumerate}\item \label{#1} #2\suspend{enumerate}}

\newtheorem{lem}{Lemma}

\newtheorem{thm}{Theorem}
\newtheorem{cor}{Corollary}

\usepackage{graphicx} 

\usepackage{array} 

\def\<{\langle}
\def\>{\rangle}

\def\R{\mathbb R}

\title{A simple integral representation of single-event scoring rules}
\author{Alexander R. Pruss}
\address{
Alexander R. Pruss
Department of Philosophy, Baylor University
}

\begin{document}
\begin{abstract}
A simple integral representation involving no derivatives or continuity assumptions is given for
proper single-event scoring rules.
\end{abstract}
\sloppy

\maketitle
A proper (accuracy, single-event) scoring rule is a pair of functions $(T,F)$ from $[0,1]$ to $[-\infty,0)$
that are finite on $(0,1)$ and satisfy the propriety inequality for all $p,q\in [0,1]$
$$
	p T(p) + (1-p) F(p) \ge p T(q) + (1-p) F(q)
$$
with the convention that $0 \cdot (-\infty)=0$~\cite{GR}. It is well-known that this
implies that $T$ and $F$ are monotone non-decreasing and non-increasing, respectively~\cite{Schervish}.

A number of representations of scoring rules are known.
For instance, every proper scoring rule is of the form:
$$
	T(p) = G(p)+(1-p)G'(p)
$$
and
$$
	F(p) = G(p)-pG'(p)
$$
for a convex function $G$ with a subgradient $G'$~\cite[p.~343]{GR}. Likewise, Schervish has given an integral representation of
left-continuous proper scoring rules~\cite{Schervish} in terms of a measure on $[0,1]$, and a generalization to all
proper scoring scoring rules in terms of a countable sum of integrals. 

The purpose of this note is to provide a simple integral representation that does not involve any continuity, and that represents $F$ in terms of $T$.

\begin{thm}\label{thm:repr}
Let $T:[0,1]\to[-\infty,\infty)$ be monotonic non-decreasing and finite on $(0,1)$.
Fix any real constant $C$ and any $c\ge 0$. Let
$$
	F(x) = C-\frac{xT(x)}{1-x}+\int_{1/2}^x \frac{T(u)}{(1-u)^2} \, du,
$$
for $x<1$. If $T$ is continuous at $1$, let $F(1)=-c+\lim_{x\to 1-} F(x)$, and otherwise
let $F(1)=-\infty$. Then the limit is defined if $T$ is continuous at $1$, and in any case $(T,F)$ is proper.

Conversely, any function $F$ on $[0,1]$ that is finite except perhaps at $1$ and such that
$(T,F)$ is proper satisfies the above definition for some real $C$ and $c\ge 0$. 
\end{thm}

A pair of functions $(T,F)$ is a proper scoring rule if and only if $(F^*,T^*)$ is a proper scoring rule,
where $g^*(x)=g(1-x)$. Thus:

\begin{cor}\label{cor:F}
Let $F:[0,1]\to[-\infty,\infty)$ be monotonic non-increasing and finite on $(0,1)$.
Fix any real constant $C$ and any $c\ge 0$. Let
$$
	T(x) = C-\frac{(1-x)F(x)}{x}+\int_{x}^{1/2} \frac{F(u)}{(u)^2} \, du,
$$
for $x>0$. If $F$ is continuous at $0$, let $T(0)=-c+\lim_{x\to 0+} T(x)$, and otherwise
let $T(0)=-\infty$. Then the limit is defined if $F$ is continuous at $1$, and in any case $(T,F)$ is proper.

Conversely, any function $T$ on $[0,1]$ that is finite except perhaps at $1$ and such that
$(T,F)$ is proper satisfies the above definition for some real $C$ and $c\ge 0$. 
\end{cor}

\begin{cor}
Any monotonic non-decreasing function $T$ from $[0,1]$ to $[-\infty,\infty)$ that's finite on $(0,1)$ is the first 
element of a proper scoring rule $(T,F)$. Any monotonic non-increasing function $F$ from $[0,1]$ to 
$[-\infty,\infty)$ that's finite on $(0,1)$ is the second element of a proper scoring rule $(T,F)$.
\end{cor}

\begin{cor}
If $(T_1,F_1)$ and $(T_2,F_2)$ are proper scoring rules that are continuous at $0$ and $1$, then 
$(T_1-T_2,F_1-F_2)$ is a proper scoring rule if and only if either $T_1-T_2$ is non-decreasing or
$F_1-F_2$ is non-increasing.
\end{cor}

The following lemma is likely folklore.

\begin{lem}\label{lem:unique}
Suppose $(T,F)$ and $(T,K)$ are both proper. Then $F$ and $K$ differ by a constant on $[0,1)$.
\end{lem}

\begin{proof}
By the convexity and subgradient representation, there is a a convex function $G$
on $[0,1]$ with a subgradient $G'$ which is real-valued, except perhaps at $0$ and $1$, such that
$$
	T(x)=G(x)+(1-x)G'(x)
$$
and
$$
	F(x)=G(x)-x G'(x).
$$
Then there is also a convex $H$ with a subgradient $H'$ (real except perhaps at endpoints) such that 
$T(x)=H(x)+(1-x)H'(x)$ and $K(x)=H(x)-xH'(x)$. 
Note that if $H_1(x)=H(x)+c(1-x)$, then $H_1'(x)=H'(x)-c$ is a subgradient of $H_1$, and
$$
\begin{aligned}
	H_1(x)+(1-x)H_1'(x)&=H(x)+c(1-x)+(1-x)(H'(x)-c)\\
		&=H(x)+(1-x)H'(x)=T(x)
\end{aligned}	
$$
while
$$
	H_1(x)-xH_1'(x)=H(x)+c(1-x)-x(H'(x)-c)=H(x)-xH'(x)+c=K(x)+c.
$$
Choosing $c$ such that $H(1/2)+(1/2)c=G(1/2)$, and replacing $K(x)$ with
$K(x)+c$, $H(x)$ with $H(x)+c(1-x)$ and $H'(x)$ with $H'(x)-c$, we can assume
that $H(1/2)=G(1/2)$.
Then for all $x\in (0,1)$:
$$
	\Delta(x) = \int_{1/2}^x \Delta'(x) \, dx,
$$
since convex functions are absolutely continuous on compact intervals and hence equal to the
Lebesgue integral of their derivatives, which when defined are equal to the subgradient. Moreover,
$$
	\Delta(x)+(1-x)\Delta'(x)=0
$$
almost everywhere since $G(x)+(1-x)G'(x)=H(x)+(1-x)H'(x)$. Thus:
$$
	\Delta(x) = \Delta(1/2)+\int_{1/2}^x \frac{\Delta(x)}{1-x} \, dx.
$$
Since $\Delta(x)$ is continuous, the right-hand side is differentiable in $[0,1)$, and hence we have
$$
	\Delta'(x) = \frac{\Delta(x)}{1-x}
$$
everywhere on $[0,1)$ and so
$$
-\Delta(x)+(1-x)\Delta'(x)=0. 
$$
Now, the derivative of $(1-x)\Delta(x)$ is $-\Delta(x)+(1-x)\Delta'(x)=0$, so $(1-x)\Delta(x)$ 
is constant on $[0,1)$. Since $(1-x)\Delta(x)=0$ at $x=1/2$, it must be equal to $0$ everywhere.\footnote{I am 
grateful to Gemini Pro for some of the main ideas in this proof.}
\end{proof}

\begin{proof}[Proof of Theorem~\ref{thm:repr}]
First consider the case where $T(x)$ is an indicator function of an interval $A \subseteq (0,1]$ containing
$1$, where $1_A(x)$ is $1$ if
$x\in A$ and $0$ if $x\notin A$. Let:
$$
	G_A(x) = -\frac{xT(x)}{1-x}+\int_0^x \frac{T(u)}{(1-u)^2} \, du
$$
for $x<1$. For $x \in (0,1)-A$ we then have $F_A(x)=0$. For $x \in A\cap (0,1)$, we have $x\ge a$ and:
$$
\begin{aligned}
	G_A(x) &= -\frac{x}{1-x}+\int_a^x (1-u)^{-2} \,du \\
		 &= -\frac{x}{1-x}+\frac{1}{1-x}-\frac{1}{1-a} \\
		 &= -\frac{a}{1-a}.
\end{aligned}
$$		 
Thus $G_A(x) = -(a/(1-a))\cdot 1_A(x)$ for $x<1$. If $A=\{ 1 \}$, then $F_A$ is discontinuous at $1$,
and let $G_A(1)=-\infty$. Otherwise, $F_A$ is continuous at $1$, and let $-G_A(1)=(a/(1-a))\cdot 1_A(1)$, 
with the convention that this is everywhere $0$ if $A=\varnothing$. It is easy to check that 
$(1_A,G_A)$ is proper.

Now let $T_B = -1_B$ where $B$ is of the form $[0,b]$ and $[0,b)$ for $b\le 1$. Let:
$$
	F_B(x) = -\frac{xT(x)}{1-x}+\int_{1/2}^x \frac{T(u)}{(1-u)^2} \, du
$$
for $x<1$. Let $F_B(1)=-\infty$ if $B=[0,1)$ and otherwise let $F_B(1)=\lim_{x\to 1-} F_B(x)$.
If $b<1$, then $T_B$ differs by a constant from $1_{B^c}$ and $F_B$ by a constant from $G_{B^c}$. Thus
$(T_B,F_B)$ is also proper by what we proved already. If $B=[0,1]$, then $T_B$ and $F_B$ are both constant, and
$(T_B,F_B)$ are proper. If $B=[0,1)$, then $F_B$ is $-\infty$ at $1$ and on $[0,1)$ it's constant. In that
case it is easy to verify the propriety inequality
$$
	p T_B(p) + (1-p) F_B(p) \ge p T_B(x) + (1-p) F_B(x).
$$
For both sides are equal if $p,x<1$. If $p=1$, the inequality is equivalent to $p T_B(p) \ge p T_B(x)$ which
holds by monotonicity. If $x=1$ and $p<1$, the right-hand-side is $-\infty$. So, in all cases we have
$(T_B,F_B)$ proper.

Now given any non-decreasing $T(x)$ bounded below, without loss of generality assume $T$ is non-positive
(subtract a constant if necessary). Let $B_t=\{ x : T(x) \le t \}$ be a level set of $T$. Then $B_t$ is of the form 
$[0,b]$ or $[0,b)$ and:
\begin{equation}\label{eq:Tint}
	T(x) = \int_{-\infty}^0 T_{B_t} \, dt.
\end{equation}
Let $F$ be as in the statement of the Lemma in the special case where $C=c=0$. Thus:
$$
	F(x) = -\frac{xT(x)}{1-x}+\int_{1/2}^x \frac{T(u)}{(1-u)^2} \, du,
$$
and $F(1)$ is $-\infty$ if $T$ is continuous at $1$ and is $\lim_{x\to1-} F(x)$ otherwise.
We will show that $F(T,F)$ is proper.

For $x<1$ we have:
\begin{equation}\label{eq:Fint}
\begin{aligned}
	F(x) &= -\frac{xT(x)}{1-x}+\int_{1/2}^x \frac{T(u)}{(1-u)^2} \, du \\
		 &= -\frac{x}{1-x}\int_{-\infty}^0 T_{B_t} \, dt + \int_{1/2}^x \int_{-\infty}^0 \frac{T_{B_t}}{(1-u)^2} \, dt\, du \\
		 &= \int_{-\infty}^0 \left(-\frac{xT_{B_t}}{1-x}  + \int_{1/2}^x \frac{T_{B_t}}{(1-u)^2} \, du \right) \, dt\\
		 &= \int_{-\infty}^0 F_{B_t}(x) \, dt,
\end{aligned}		 
\end{equation}
where the order of integration was swapped by Tonelli's theorem as $T\le 0$. By propriety of
$(T_{B_t},F_{B_t})$ we have:
$$
	p T_{B_t}(p) + (1-p) F_{B_t}(p) \ge 
	p T_{B_t}(x) + (1-p) F_{B_t}(x),
$$
for $p,x<1$, 
and integrating from $-\infty$ to $0$ with respect to $t$, by \eqref{eq:Tint} and \eqref{eq:Fint} we get
\begin{equation}\label{eq:propriety}
	p T(p) + (1-p) F(p) \ge 
	p T(x) + (1-p) F(x).
\end{equation}

Note that $F_{B_t}$ differs by a constant from $G_{B_t^c}$ on $[0,1)$, and hence is non-increasing on $[0,1)$.
By \eqref{eq:Fint} this is also true of $F$, so it has a limit at $1$. Define $F$ to be that limit if
$T$ is continuous at $1$ and to be $-\infty$ otherwise. In the case where $T$ is continuous at $1$, we get
\eqref{eq:propriety} for $p=1$ and $x<1$ by taking the limit as $p\to 1-$, and for $p<1$ and $x=1$ by
taking the limit as $x\to 1-$. The case of $p=x=1$ is trivial. What remains is the case where $T$ is
not continuous at $1$ and either $p=1$ or $x=1$. If $p=1$, then \eqref{eq:propriety} follows
from the monotonicity of $T$, and if $p<1$ but $x=1$, then \eqref{eq:propriety} is trivial when
$F(1)=-\infty$.

Now fix $C\in\R$ and $c\ge 0$. Let $K(x)=F(x)+C$ for $x<1$ and $K(1)=F(1)+C-c$. I claim that $(F,K)$ is proper.
For \eqref{eq:propriety}
remains true with $K$ in place of $F$ if $p$ and $x$ are in $[0,1)$. If $p=1$, then $(1-p)K(p)=(1-p)F(p)=0$,
and so \eqref{eq:propriety} is also true with $K$ in place of $F$. The remaining case to check where $p<1$ and
$x=1$. By \eqref{eq:propriety} we have:
$$
	p T(p) + (1-p) F(p) \ge p T(1) + (1-p) F(1).
$$
But $F(p)=K(p)-C$ and $F(1)\ge K(1)-C$, so we can indeed put $K$ in place of $F$.

It remains to show that for any $K$ such that $(T,K)$ is proper and $K$ is finite except perhaps at $1$
there are $C$ and $c$ such that $K(x)=F(x)+C$ for $x<1$ and $K(1)=F(1)+C-c$. The $x<1$ part is 
Lemma~\ref{lem:unique}. It remains to show that $K(1)\le F(1)+C$. Suppose first that $T$ is discontinuous
at $1$. By propriety of $(T,K)$, for $p<1$ we have:
$$
	p (T(p)-T(1) \ge (1-p)(K(1)-K(p)).
$$
Thus,
$$
	K(1)-K(p) \le \frac{p}{1-p} (T(p)-T(1)).
$$
Since $T$ is non-decreasing and discontinuous at $1$, the limit of the right hand side as $p\to 1-$ is
$-\infty$, and hence $K(1)=-\infty=F(1)+C$. 

Now suppose $T$ is continuous at $1$, so $K(1)=F(1)+C=C+\lim_{p\to 1-} F(p)$. But
$K(1)\le \lim_{p\to 1-}$ by monotonicity, so $K(1)=F(1)+C-c$ for some positive $c$.
\end{proof}

\end{document}